\documentclass[12pt]{amsart}

\usepackage{amsmath,amsthm,amsfonts,amssymb,amscd}
\usepackage[all]{xy}
\usepackage{sseq}

\theoremstyle{plain}
\newtheorem{thm}{Theorem}[section]
\newtheorem{prop}[thm]{Proposition}
\newtheorem{lem}[thm]{Lemma}
\newtheorem{cor}[thm]{Corollary}

\theoremstyle{definition}

\newtheorem{remark}[equation]{Remark}

\theoremstyle{remark}

\numberwithin{equation}{section}

\newcommand{\Z}{\mathbb{Z}}

\newcommand{\Q}{\mathbb{Q}}
\newcommand{\R}{\mathbb{R}}

\newcommand{\C}{\mathbb{C}}

\newcommand{\ti}{\tilde}

\renewcommand{\phi}{\varphi}

\newcommand{\im}{\mathrm{im}}

\DeclareMathSymbol{\sdp}{\mathbin}{AMSb}{"6F}

\newcommand{\E}{\mathcal{E}}

\newcommand{\ignore}[1]{}

\newcommand{\rank}{\mathrm{rank}}
\begin{document}
\baselineskip=16pt

\begin{abstract}
Consider a Riemann surface $X$ of genus $g \geq 2$ equipped with an antiholomorphic involution $\tau$. This induces a natural involution on the moduli space $M(r,d)$ of semistable Higgs bundles of rank $r$ and degree $d$.  If $D$ is a divisor such that $\tau(D) = D$, this restricts to an involution on the moduli space $M(r,D)$ of those Higgs bundles with fixed determinant $\mathcal{O}(D)$ and trace-free Higgs field. The fixed point sets of these involutions $M(r,d)^{\tau}$ and $M(r,D)^{\tau}$ are $(A,A,B)$-branes introduced by Baraglia-Schaposnik \cite{BS}. In this paper, we derive formulas for the mod 2 Betti numbers of $M(r,d)^{\tau}$ and $M(r,D)^{\tau}$ when $r=2$ and $d$ is odd. In the course of this calculation, we also compute the mod 2 cohomology ring of $SP^m(X)^{\tau}$, the fixed point set of the involution induced by $\tau$ on symmetric products of the Riemann surface.   
\end{abstract}

\title{Symmetric products of a real curve and the moduli space of Higgs bundles}
\address{ 	Department of Mathematics and Statistics,
Memorial University of Newfoundland,
St. John's, NL,
Canada,
A1C 5S7,
tbaird@mun.ca}
\author{Thomas John Baird}
\keywords{Moduli spaces of Higgs bundles, symmetric products of a curve, spectral sequences, Betti numbers, anti-holomorphic involutions}

\maketitle

\section{Introduction}

Let $X$ denote a compact, connected Riemann surface with canonical bundle $K$.  A \emph{Higgs bundle} $(E,\Phi)$ over $X$ consists of a holomorphic vector bundle $E$ over $X$ and a section $\Phi \in H^0(X, Hom(E, E \otimes K))$ called the \emph{Higgs field}.  A Higgs field is called \emph{stable} if all proper vector subbundles $F \subset E$ such that $\Phi(F) \subseteq F \otimes K$ satisfy $deg(F)/ rank(F) \leq deg(E)/rank(E)$.  In \cite{H}, Hitchin constructed the moduli space $M(r,d)$ of semistable Higgs bundles. We will always assume that $r$ and $d$ are coprime and $X$ has genus $g \geq 2$, so $M(r,d)$ is non-singular.

Fix a divisior $D \in Div(X)$.  Define $M(r, D)$ to be the subvariety of $M(r, d)$ of Higgs bundles $(E, \Phi)$ for which $\wedge^r E \cong \mathcal{O}(D)$ and $tr(\Phi)=0$.  Both $M(r,d)$ and $M(r, D)$ admit a complete hyperk\"ahler metric: a Riemannian metric $g$ which is K\"ahler with respect to three different complex structures $I, J,K$ that satisfy the quaternionic relations. We denote by $\omega_I, \omega_J, \omega_K$ the corresponding K\"ahler forms.

Suppose that $X$ admits an anti-holomorphic involution $\tau$ and call $(X,\tau)$ a \emph{real curve}. This induces an involution on $M(r,d)$ (which we also call $\tau$) sending a pair $(E,\Phi)$ to $\tau(E,\Phi) := (\tau(E), \tau(\Phi))$, where $ \tau(E) = \tau^*\overline{E}$ is the conjugate pull-back and $\tau(\Phi)$ is the composition 
$$ \xymatrix{ \tau^*\overline{E} \ar[rr]^{(\tau^*)^{-1}} && E \ar[r]^{\Phi} & E \otimes K \ar[r]^{\tau^*}  & \tau^*(\overline{E} \otimes \overline{K}) \ar[r]^{\cong} &  \tau^*(\overline{E} )\otimes K } $$  
where we have used the natural isomorphism $K \cong \tau^* \overline{K}$ determined by the fact that $\tau$ is anti-holomorphic. If $D \in Div(X)$ is a \emph{real divisor} in the sense that $\tau(D) = D$, then $\tau$ restricts to an involution of $M(r,D)$.  

This involution was considered by Baraglia-Schaposnik \cite{BS}. It preserves the hyperk\"ahler metric, is anti-holomorphic with respect to $I,J$ and is holomorphic with respect to $K$. Consequently, the fixed point sets of the involutions $M(r,d)^\tau$ and $M(r,D)^\tau$ are real and Lagrangian with respect to $I,J$ and are complex and symplectic with respect to $K$.  Such a submanifold is called an $(A,A,B)$-brane, and plays a role in the Kapustin-Witten approach to geometric Langlands \cite{KW,GW, BN}. In this paper, we derive formulas computing the mod 2 Betti numbers of $M(r,d)^{\tau}$ and $M(r,D)^{\tau}$ in case rank $r=2$ and degree $d$ is odd.  

More generally, for any $\theta \in [0,2 \pi)$ we may consider the involution $\tau_{\theta}$ defined by $\tau_{\theta} (E ,\Phi) = (\tau(E) , e^{i\theta} \tau(\Phi))$.  This involution is anti-holomorphic with respect to $I$ and $cos(\theta/2) J - sin(\theta/2) K$ and holomorphic with respect to $sin(\theta/2)J +cos(\theta/2) K$. In particular, the $\tau_{\pi}$ fixed point set is an (A,B,A)-brane in the sense of Kapustin-Witten \cite{KW}. Since
$$\tau_{\theta}  := e^{i\theta} \tau = e^{i \theta/2} \tau e^{-i \theta/2},  $$
the diffeomorphism type of the fixed point set of $\tau_{\theta}$ is independent of $\theta$ and it suffices for our purposes to work with $\tau_0 = \tau$.

\subsection{Outline of the proof}\label{Outline of the proof}

There is natural $\C^*$-action on $M(r,D)$ by scaling the Higgs field. Hitchin observed that the restricted $U(1)$-action is Hamiltonian with respect to the symplectic structure $\omega_I$, with proper moment map $\mu: M(r,D) \rightarrow \R,$ $$\mu(E,\Phi) = \| \Phi \|_{L^2}^2.$$ Therefore, by a theorem of Frankel \cite{F}, the function $\mu$ is a perfect Morse-Bott function with respect to rational coefficients and the critical points of $\mu$ coincide with the $U(1)$-fixed points. This means we have an equality

 $$  P_t^{\Q}(M(r,D)) = \sum_{F \text{ component of } M(r,D)^{U(1)}} P_t^{\Q}( F) t^{2 d_F}  $$
where $P_t^{\mathbb{K}}(Y) := \sum_{i=0}^{\infty} \dim(H^i(Y;\mathbb{K})) t^i$ is the Poincar\'e series and $2 d_F$ is the Morse index of the path component $F$ (which is necessarily even because the negative normal bundles are symplectic). This reduces the calculation of the rational Betti numbers of $M(r,D)$ to calculating the Betti numbers of the fixed point components $F$ and their Morse indices $2 d_F$. This was carried out for rank $r=2$ by Hitchin \cite{H} and for rank $r=3$ by Gothen \cite{G}.

Similar considerations apply to compute mod 2 Betti numbers of $M(r,D)^{\tau}$. The involution is compatible with the $U(1)$-action in the sense that $ e^{i \theta} \circ \tau = \tau \circ e^{-i \theta}$ and $\mu \circ \tau = \mu$. In this circumstance, a theorem of Duistermaat \cite{D,BGH} tells us that the restriction of $\mu$ to $M(r,D)^{\tau}$ is a perfect Morse-Bott function with respect to mod 2 coefficients. The set of critical points of $\mu$ restricted to $M(r,D)^{\tau}$ coincides with $M(r,D)^{\tau} \cap M(r,D)^{U(1)}$ and the Morse indices are halved (since they compute the dimension of Lagrangian vector subbundles of symplectic vector bundles). Consequently, we obtain the formula
\begin{equation}\label{PPMform}
 P_t^{\Z_2}(M(r,D)^{\tau}) = \sum_{F \text{ component of } M(r,D)^{U(1)}} P_t^{\Z_2}( F^{\tau} ) t^{d_F}.
\end{equation}
Thus to compute the mod 2 Betti numbers of $M(r,D)^{\tau}$ it remains only to compute those of $F^{\tau}$. 

The Morse function $\mu$ is globally minimized on $M(r,D)$ exactly when the Higgs field vanishes.  Therefore the minimizing set of $\mu$ on $M(r,D)$ is identifed with the \emph{moduli space of stable vector bundles} $N(r,D)$ of rank $r$ and determinant $\mathcal{O}(D)$. The global minimizing set of $\mu$ restricted to $M(r,D)^{\tau}$ is consequently identified with $N(r,D)^{\tau}$, the \emph{moduli space of Real vector bundles} of rank $r$ and determinant $\mathcal{O}(D)$. This moduli space was introduced in \cite{BHH, S} and its mod 2 Betti numbers were computed in \cite{B, B2, LS}.

Restrict now to the case where the rank $r = 2$.  Hitchin shows that the remaing $U(1)$-fixed points are represented by pairs $(E, \Phi)$ of the form

\begin{equation*}
 \E  = L \oplus (L^* \otimes \mathcal{O}(D)),~~~~  \Phi =  \begin{bmatrix}
    0       & 0  \\
    \phi     &  0
\end{bmatrix}
\end{equation*}
where $\phi \in H^0( L^{-2} \otimes K(D) )$.  The fixed point components are identified with pullbacks of the form
\begin{equation}\label{pullbackdiagram}
 \xymatrix{  F_l \ar[r] \ar[d] & Pic_l(X) \ar[d]^{sq} \\
SP^{m}(X)  \ar[r]^{aj} & Pic_m(X)  }
\end{equation}
where $SP^m(X)$ is the $m$-fold symmetric product of $X$, $aj$ is the Abel-Jacobi map, $sq$ is the map sending $[L]$ to $[L^{-2} \otimes K(D)]$, $m = 2g-2 - 2l+d$, and $l$ ranges between $1 \leq l \leq g-1$. Here, $sq$ is a $2^{2g}$-fold covering map that can be identified with the squaring map under an appropriate translation $ Pic_l(X) \cong Pic_m(X) $ to the Jacobian $Jac(X):=Pic_0(X)$ which is isomorphic to $U(1)^{2g}$ as a Lie group. 

The diagram (\ref{pullbackdiagram}) is equivariant with respect to the induced $\tau$-actions and we identify $F_l^{\tau}$ with the pull-back of the restriction to $\tau$-fixed points

\begin{equation}\label{pullbackreal}
 \xymatrix{  F_l^\tau \ar[r] \ar[d] & Pic_l(X)^\tau \ar[d]^{sq} \\
SP^{m}(X)^\tau  \ar[r]^{aj} &Pic_m(X)^\tau  .} 
\end{equation}

Thus we are reduced to computing mod 2 Betti numbers of covering spaces of path components of $SP^{m}(X)^{\tau}$. We first compute the mod 2 cohomology ring of path components of $SP^{m}(X)^{\tau}$ in \S \ref{Symmetric products of a real curve}. We then employ an Eilenberg-Moore spectral sequence argument to compute the Betti numbers of the cover in \S \ref{Real Abel-Jacobi and covering spaces}.  In \S \ref{Higgs bundle Poincare series} we input the result into equation (\ref{PPMform}) to calculate the Betti numbers of $M(2,D)^{\tau}$, from which the Betti numbers of $M(2,d)^{\tau}$ are easily determined.

\section{Topology of real curves and Picard groups}\label{Topology of real curves}

We summarize some material from Gross-Harris \cite{GH}. Given a real curve $(X,\tau)$, the fixed point set $X^{\tau}$  is a union of circles that we call \emph{real circles}. Compact, connected real curves $(X,\tau)$ are classified by three invariants $(g,n,a)$, where $g$ is the genus, $n$ is the number of real circles, and either $a=1$ if $X \setminus X^{\tau}$ has connected, or $a=0$ if it is disconnected. The range of possible values is
 $$\{ (g,n,a)~|~1-a \leq n \leq g+1-a,\text{ and if }  a =0 \text{, then } g-n \equiv1~mod~2  \}.$$
 
A Real line bundle $(L,\tilde{\tau})$ over $(\Sigma,\tau)$ is a line bundle equipped with an isomorphism $\tilde{\tau}: L \xrightarrow{\sim} \tau^* \overline{L}$. The fixed point set $L^{\tilde{\tau}} \rightarrow \Sigma^{\tau}$ is an $\R^1$-bundle with first Stiefel-Whitney class $w_1(L^{\tau}) \in H^1(\Sigma^{\tau}; \Z_2)$. If $X^{\tau}$ is non-empty, then $L \cong \mathcal{O}(D)$ for some real divisor $D$ and we define $w_1(D) = w_1( \mathcal{O}(D)^{\tau})$.
If $C \subset X^{\tau}$ is a real circle, then the following are equivalent
\begin{enumerate}
\item[(i)]  $\mathcal{O}(D)^{\tilde{\tau}}|_C$ is non-orientable,
\item[(ii)]  $w_1(D)(C) = 1$,
\item[(iii)] $D$ has odd degree supported on $C$.
\end{enumerate}
We call $C$ odd with respect to $D$ if it satisfies these equivalent conditions. If $k$ is the number of odd circles of $D$ then
\begin{equation}\label{d=w=k}
\deg(D)  \equiv w_1(D)(\Sigma^{\tau}) \equiv k~mod~2.
\end{equation}

Assume now that $(X,\tau)$ has $n\geq 1$ real circles. Then $Pic_0(X)^{\tau} \cong (\Z/2)^{n-1}\times (S^1)^g$ as a Lie group. For general $m \in \Z$, $Pic_m(X)^{\tau}$ is a torsor for $Pic_0(X)^{\tau}$ yielding a diffeomorphism $Pic_m(X)^{\tau} \cong (\Z/2)^{n-1}\times (S^1)^g$. The path components are classified by $w_1([D]) = w_1(D)$ which must satisfy $w_1([D]) (X^{\tau}) \equiv m ~mod~2.$ The Abel-Jacobi map restricts to a map $$aj: SP^m(X)^{\tau} \rightarrow Pic_m(X)^{\tau}$$ which is injective on $\pi_0$ and surjective on $\pi_0$ if $m \geq n-1$. 

If $X^{\tau} = \emptyset$, then the behaviour of $Pic_m(X)^{\tau}$ is complicated by the existence of \emph{quaternionic line bundles} \cite{BHH}. In particular, there exist divisor classes $[D] $ that are fixed by $\tau$ but are not represented by a real divisor. This issue is unimportant for our purposes, because our condition that $\deg(D)$ is odd will force $(X,\tau)$ to have real points by (\ref{d=w=k}).

\section{ Symmetric products of a real curve}\label{Symmetric products of a real curve}

Let $(X,\tau)$ be a real curve of genus $g$. For $m\geq 0$, the symmetric product $SP^m(X) := (X\times ... \times X)/S_m$ is a $m$-dimensional complex manifold, which is identified with the set of effective divisors of degree $d$ (by convention $SP^0(X)$ is a point). The involution $\tau$ acts on $SP^m(X)$ and our goal in this section is to calculate the mod 2 cohomology ring of the path components of the fixed point set $ SP^m(X)^{\tau}$.

\subsection{Symmetric products and the case $X^{\tau} = \emptyset$}

We review the cohomology theory of symmetric products of surfaces and prove a few results for later application. See \cite{BGZ} for a nice survey of this topic. All cohomology is taken with coefficients $\Z_2$ unless otherwise indicated.

Consider the sequence of inclusions $$X = SP^1(X) \hookrightarrow SP^2(X) \hookrightarrow ...$$ defined by adding a basepoint at each step. The colimit of this sequence, denoted $SP^{\infty}(X)$, is the free abelian monoid on $X$.  The Dold-Thom Theorem yields a homotopy equivalence
$$ SP^{\infty}(X) \sim K( H^1(X;\Z), 1) \times K(H^2(X;\Z), 2) \sim Jac(X) \times \C P^{\infty} $$
where $Jac(X)$ is the Jacobian of $X$. The cohomology ring is a graded symmetric algebra
 $$ H^*(SP^{\infty}(X)) \cong \wedge(x_1,...,x_g, y_1,...,y_g) \otimes S(\eta)$$ 
where $x_1,...,x_g, y_1,...,y_g$ have degree one and $\eta$ has degree two. The following formulas hold over $\Z$, but we work always over $\Z_2$. Macdonald \cite{M} proved:

\begin{thm}\label{McDThm}
The inclusion $SP^m(X) \hookrightarrow SP^{\infty}(X)$ induces a surjection in cohomology $$ \wedge( x_1,...,x_g, y_1,...,y_g)\otimes S(\eta)  \rightarrow  H^*(SP^m(X))$$ with relations generated by 
\begin{equation}\label{McDrelations}
 x_{i_1}...x_{i_a} y_{j_1}...y_{j_b} (x_{k_1}y_{k_1} - \eta)... (x_{k_c}y_{k_c} - \eta) \eta^q =0 .
\end{equation}
for all choices $i_1,..,i_a, j_1,...,j_b, k_1,...,k_c$ of distinct integers from 1 to g inclusive such that $a +b+2c+q = m+1$.
The $\Z_2$-Poincar\'e series and Euler characteristic statisfy
$$ P_t(SP^m(X))) = \sum_{i=0}^{min(2g,m)} { 2g \choose i} (t^i+t^{i+2}+...+t^{2m-i}) = \sum_{i=0}^{min(2g,m)} { 2g \choose i}  \frac{ t^{2m+2-i}-t^i}{t^2-1} $$ 
$$ \chi( SP^m(X)) =  (-1)^m{ 2g-2 \choose m}  $$
\end{thm}

We will replace $SP^{\infty}(X)$ with a finite dimensional approximation 
\begin{equation}\label{finiteDoldThom}
 SP^m(X) \rightarrow Pic_m(X) \times \C P^m
\end{equation}
inducing a surjection on cohomology with the same generators and relations. The first coordinate function of (\ref{finiteDoldThom}) is the Abel-Jacobi map, sending $D$ to $[D]$. For the second, define a continuous map 
\begin{equation}\label{collapsemap}
X \rightarrow  S^2 \cong \C P^1
\end{equation} 
by collapsing the complement of a disk in $X$ to a point (obviously this is not holomorphic). The second coordinate function is the induced map $SP^m(X) \rightarrow SP^m(\C P^1) \cong P( Sym^m(\C^2)) \cong \C P^m$. If $X^{\tau}$ is non-empty, then map (\ref{finiteDoldThom}) can be chosen $\tau$-equivariant with respect to the standard involution on $\C P^m$. 

If $X^{\tau}$ is empty, then every element of $SP^m(X)^{\tau}$ must be a sum of conjugate pairs of points. But conjugate pairs of points are in one-to-one correspondence with points in the quotient space $X/\tau$, which is a non-orientable surface homeomorphic to the $(g+1)$-fold connected sum $N_g := \R P^2 \# ... \# \R P^2$. Thus,

\begin{prop}\label{Nofixed}
Suppose $(X,\tau)$ is a real curve of genus g and $X^{\tau} = \emptyset$. Then 
$$ SP^m(X)^{\tau} \cong \begin{cases} \emptyset & \text{ if $m$ is odd} \\ SP^{m/2}(X/\tau) = SP^{m/2}(N_g) & \text{ if $m$ is even} .\end{cases} $$
\end{prop}

The cohomology ring of $SP^{m}(N_g) $ was computed by Kallel and Salvatore \cite{KS}.

\begin{thm}\label{nonerientablesymm}
The inclusion $SP^m(N_g) \hookrightarrow SP^{\infty}(N_g) \sim K(H_1(N_g;\Z), 1) \sim (S^1)^g\times \R P^{\infty}$ induces a surjection in cohomology $$ \wedge( h_1,...,h_g)\otimes S(w)  \rightarrow  H^*(SP^m(N_g);\Z_2)$$ where $h_1,...,h_g,w$ all have degree one.  Let $f_i = h_i+w$ for $i=1,...,g$, let $f_{g+1} = w$ and $b=w^2$. Then the relations are generated by 
\begin{equation}\label{KSrelations}
 f_{i_1} ... f_{i_r} b^t =0.
\end{equation}
for all choices $i_1,..,i_r$ of distinct integers from $1$ to $g+1$ inclusive such that $r+t = m+1$.
The $\Z_2$-Poincar\'e series and Euler characteristic statisfy
$$ P_t(SP^m(N_g))) = \sum_{i=0}^{min(g, m)} { g \choose i} (t^i+t^{i+1}+...+t^{2m-i}) =  \sum_{i=0}^{min(g, m)} { g \choose i} \frac{t^{2m-i+1}-t^i}{t-1} $$ 
$$ \chi( SP^m(N_g)) =  (-1)^m{ g-1 \choose m}  $$
\end{thm}

\begin{remark}\label{n=o}
When $g\geq 1$, one can use the identity $ { 2g \choose i } =  { 2g-1 \choose i}  +  { 2g-1 \choose i-1} $, to get $P_t(SP^m(N_{2g-1})) = P_t(SP^m(\Sigma_g))$. This can also be deduced from the equality $P_t(N_{2g-1}) = P_t(\Sigma_g)$ by a Theorem of Dold \cite{Do}.
\end{remark}

It will be useful to have a geometric understanding of the generator $b$ in Theorem \ref{nonerientablesymm}. Let $C \subset N_g$ be a loop with non-orientable tubular neighbourhood $U$. Taking the quotient of $N_g$ by the complement of $U$ produces a manifold homeomorphic to a real projective plane $N_g / (N_g \setminus U) \cong \R P^2$. Taking symmetric products determines a map $q_m: SP^m(N_g) \rightarrow SP^m(\R P^2) \cong \R P^{2m}$ where the homeomorphism $SP^m(\R P^2) \cong \R P^{2m}$ is due to Dupont-Lusztig \cite{DP}. Recall that if $m\geq 1$, then $H^*(\R P^{2m})$ is a truncated polynomial ring generated by an element $w \in H^1(\R P^{2m})$. 

\begin{lem}
In terms of the presentation in Theorem \ref{nonerientablesymm}, $b \in H^2( SP^m(N_g))$ is the image of $w^2 \in H^*(\R P^{2m})$ under $q^*_m$.
\end{lem}
\begin{proof}
This is trivially true if $m=0$, so assume $m\geq1$. It is easy to check that $q_m: SP^m(N_g) \rightarrow SP^m(\R P^2) $ is a map of degree one, so by Poincar\'e duality it induces an injection in cohomology (remember we are working with $\Z_2$-coefficients). Thus $q_m^*(w^2) = q_m^*(w)^2$ is a non-zero square in $H^2( SP^m(N_g))$. But the only non-zero square in $H^2(SP^m(N_g))$ is $b$ so $q_m^*(w^2)=b$.
\end{proof}

Suppose that $\Sigma = \Sigma_g$ is a orientable surface, and $N = N_{2g+k-1}$ is the non-orientable surface obtained by performing a real blow up at $k$ different points of $\Sigma$. We have a quotient map $bd: N \rightarrow \Sigma$ inducing maps on symmetric products $bd_m : SP^m(N) \rightarrow SP^m(\Sigma)$.

\begin{lem}\label{c to b}
In terms of the presentations in Theorem \ref{McDThm} and Theorem \ref{nonerientablesymm}, $bd_m^*(\eta)= b$.
\end{lem}
\begin{proof}
This is trivially true if $m=0$ so assume $m \geq 1$. We can arrange the maps $bd$ and $q$ so that  we have a commutative diagram
$$\xymatrix{  N \ar[r]^{bd}  \ar[d]^q & \Sigma \ar[d]^{\overline{q}}  \\  \R P^2  \ar[r]^{\overline{bd}} & S^2= \C P^1  } $$
where all maps are quotient maps, $\bar{q}$ is the quotient map defined (\ref{collapsemap}). Taking symmetric products yields
$$\xymatrix{  SP^m(N) \ar[r]^{bd_m}  \ar[d]^{q_m} & SP^m(\Sigma) \ar[d]^{\overline{q}_m}  \\  \R P^{2m} \cong SP^m(\R P^2)  \ar[r]^{\overline{bd}_m} & SP^m(S^2) = \C P^m. } $$
Since $bd_m$ has degree one, it induces an injection on cohomology. By definition, $\eta = \overline{q}^*(c)$  where $ c \in H^2(\C P^n)$ is the generator. Therefore $bd^*_m(\eta) = q^*_m( \overline{bd}_m^*(c)) = bd^*_m( w^2) = bd^*_m(w)^2$ is a non-zero square, so it must equal $b$. 

\end{proof}

Later we will encounter a case where $X$ has two path components that are transposed by $\tau$. The following is proven the same way as Proposition \ref{Nofixed}.

\begin{prop}\label{Nofixed2}
Suppose that $X$ is a disconnected union of two genus $g$ curves $X = \Sigma_g \coprod \Sigma_g$ and that $\tau$ transposes these two path components.  Then
$$ SP^m(X)^{\tau} \cong \begin{cases} \emptyset & \text{ if $m$ is odd} \\ SP^{m/2}(X/\tau) = SP^{m/2}(\Sigma_g) & \text{ if $m$ is even} .\end{cases} $$
\end{prop}

\subsection{Equivariant cohomology and localization}

Consider a $\Z/2$ representation on a $\Z_2$-vector space $V$ defined by a linear map $\phi \in GL(V)$ such that $\phi^2 = Id_V$.   The map $Id_V +\phi: V \rightarrow V$ squares to zero.  Let $Z= ker(Id_V + \phi) = V^{\Z/2}$ be the fixed point subspace, let $B = im(Id_V +\phi)$ and define $V_{triv} := Z/B$. Then $V$ decomposes into a direct sum  $\rank(V_{triv})$ many trivial representations of rank one and $\rank(B)$ many indecomposable (not irreducible!) representations of rank two given by the matrix $\begin{bmatrix} 0 & 1 \\ 1 & 0\end{bmatrix}$.  We have natural isomorphisms in group cohomology $$H^j(\Z/2; V) = \begin{cases} V^{\Z/2} & \text{ if $j=0$} \\ V_{triv} & \text{ if $j >0$} \end{cases}. $$
As a module over $H^*(\Z/2; \Z_2) = \Z_2[w]$, $H^*(\Z/2; V)$ decomposes as the direct sum of a free module $ V_{triv} \otimes_{\Z_2} \Z_2[w]$ and a torsion module concentrated in degree zero.

Suppose now that $\Z/2$ acts on a space $Y$, inducing $\Z/2$-representations $H^j(Y;\Z_2)$. The Serre spectral sequence for the Borel construction converges to $H_{\Z/2}^*( Y;\Z_2)$ and satisfies

\begin{equation}\label{SSS}
 E_2^{i,j} = \begin{cases} H^i(Y;\Z_2)^{\Z/2} & \text{ if $j=0$} \\ H^i(Y;\Z_2)_{triv} & \text{ if $j >0$}. \end{cases}
\end{equation}

In particular, this means that $E_2^{*,*}$  is isomorphic to $H^*(Y)_{triv} \otimes_{\Z_2} H^*(B\Z/2)$ modulo $H^*(\Z/2)$-torsion. Combining this with the localization theorem (see \cite{AP} \S 3.1) yields the following.

\begin{prop}\label{tightness}
Let $Y$ be a compact $\Z/2$-manifold. Then $$dim( H^*(Y;\Z_2)_{triv}) \geq dim(H^*( Y^{\Z/2};\Z_2))$$ with equality if and only if the Serre spectral sequence of the Borel construction stabilizes on page $E_2^{*,*}$. We call $Y$ ``weakly tight" when equality holds.
\end{prop}

\begin{proof}
This is an easy consequence of the localization theorem (and is proven in the special case when $\Z/2$ acts trivially on $H^*(Y)$ in \cite{AP} Theorem 3.10.4). Let $Q$ be the quotient field of $H^*(B\Z/2) =H$. Then we have an isomorphism $$ H^*_{\Z/2}(Y) \otimes_H Q \cong H^*_{\Z/2}(Y^{\Z/2}) \otimes_H Q = H^*(Y^{\Z/2}) \otimes_{\Z_2} Q$$
which has the same dimension as $H^*(Y^{\Z/2}) $. On the other hand, there is a spectral sequence
$$ E^{*,*}_2 \otimes_H Q = H^*(Y)_{triv} \otimes_{\Z_2} Q  \Rightarrow H_{\Z/2}^*(Y) \otimes_H Q .$$
The result follows.
\end{proof}

Given a real curve $(X,\tau)$ and $m \geq 0$ there is an induced involution on $Pic_m(X) \times \C P^m$ sending $( [D], [v]) \rightarrow ([\tau(D)] , [\bar{v}])$. 

\begin{prop}
Suppose $X^{\tau}$ is non-empty. Then $Pic_m(X) \times \C P^m$ is weakly tight with respect to the induced involution. 
\end{prop}

\begin{proof}

Suppose that $X^{\tau}$ consists of $n \geq 1$ real circles and define $b=n-1$. The induced action of $\tau_*$ on $H^1(X) \cong H^1(Pic_m(X))$ is described in Gross-Harris (\cite{GH} \S 4). In particular, $Id + \tau^*$ has rank $g+1 - n$, so $H^1(Pic_m(X))_{triv}$ has dimension $2 n -2 = 2b$.  It follows that we can choose a (not necessarily symplectic) basis of $H^1(Pic_m(X))$ so that 
$$ H^1(Pic_m(X)) = \Z_2\{x_1,...,x_{2b}, y_1, z_1,...,y_{g-b}, z_{g-b}\} $$ 
where the involution $\tau^*$ fixes $x_1,...,x_{2b}$ and transposes $y_i$ and $z_i$ for $i \in 1,..., g-b$. 

It follows that 

$$ H^*( Pic_m(X) \times \C P^m)_{triv} \cong    \wedge ( x_1,...,x_{2b}, d_1,...,d_{g-b}) \otimes S(c) $$
where $d_i = y_iz_i$ for $i = 1,..., g-b$.  In particular $ \dim( H^*(Pic_m(X) \times \C P^m)_{triv}) = 2^{g+b}(1+m)$.
Meanwhile, the fixed point set satisfies

$$ (Pic_m(X) \times \C P^m)^{\tau} = Pic_m(X)^\tau \times \R P^m \cong \coprod_{2^b} (S^1)^g \times \R P^m$$
so $ \dim( H^*((Pic_m(X) \times \C P^m)^\tau) =  2^{g+b}(1+m)$ and the result follows.
\end{proof}

\begin{prop}\label{equivsurje}
Suppose that $X^{\tau} \neq \emptyset$. Then $SP^m(X)$ is weakly tight and the $\tau$-equivariant map (\ref{finiteDoldThom}) induces a surjection in equivariant cohomology
$$H^*_{\Z/2}( Pic_m(X) \times \C P^m;\Z_2) \twoheadrightarrow H^*_{\Z/2}(SP^{m}(X);\Z_2). $$
\end{prop}

\begin{proof}
Consider the corresponding morphism in non-equivariant cohomology
$$ \phi^*: H^*( Pic_m(X) \times \C P^m;\Z_2) \rightarrow H^*(SP^{m}(X);\Z_2).$$
The Macdonald relations (\ref{McDrelations}) occur in degree $m+1$ or higher, so $\phi^*$ is an isomorphism in degrees less than $m+1$. Thus the induced map $\phi^*_{triv}:  H^*( Pic_m(X) \times \C P^m )_{triv} \cong H^*(SP^m( X))_{triv}  $ is also an isomorphism in degree less than $m+1$. Stiefel

The mod 2 reduction of the standard K\"ahler class on $Pic_m(X) \times \C P^m$ defined (using the earlier notation of Theorem \ref{McDThm}) by $ \omega = \eta + \sum_{i=1}^g x_i y_i \in H^2(Pic_m(X) \times \C P^m)$ is fixed by $\tau$. The mod 2 reduction of the Lefshetz maps $H^j(SP_m(X)) \rightarrow H^{2m-i}(SP_m(X))$ induced by the image of $\phi^*(\omega) \in H^2(SP^m(X))^{\Z/2}$ are isomorphisms for all $j \leq m$ yielding $\Z/2$-equivariant commutative diagrams

$$\xymatrix{ H^j(Pic_m(X) \times \C P^m) \ar[d]^{\cong} \ar[rr]^{\omega^{m-j}\cup } && H^{2m-j}(Pic_m(X) \times \C P^m ) \ar[d]  \\
H^{j}(SP^m(X)) \ar[rr]^{\cong}_{\phi^*(\omega)^{m-j}\cup  }   &&  H^{2m-j}( SP^{m}(X) )}.$$
We deduce surjections
$$ H^{2m-j}(Pic_m(X) \times \C P^m)^{\Z/2} \twoheadrightarrow H^{2m-j}(SP^m( X))^{\Z/2} .  $$
$$ H^{2m-j}(Pic_m(X) \times \C P^m)_{triv} \twoheadrightarrow H^{2m-j}(SP^m( X))_{triv} .  $$

Given Proposition \ref{SSS}, this implies that (\ref{finiteDoldThom}) induces a surjection on page 2 of the Serre spectral sequences. Since the spectral sequence for $Pic_m(X) \times \C P^m$ collapses, it must also collapse for  $SP^m(X)$ and the result follows.
\end{proof}

\begin{remark}
For $m\geq 1$, it follows that $\dim( H^*(SP^m(X)^{\tau})) = \dim( H^*(SP^m(X)))$ if and only if $X^{\tau}$ is a union of $g+1$ real circles, the maximum allowed by Harnack's inequality (also Proposition \ref{tightness}). In the language of Biswas-D'Mello \cite{BD}, the means that $SP^m(X)$ is an M-variety if and only if $X$ is an M-curve. The if direction was proven in  \cite{BD} when $m=2,3$ or $m\geq 2g-1$, and extended to all $g$ by Franz \cite{MF}.
\end{remark}

\begin{prop}
The restriction of (\ref{finiteDoldThom}) to fixed points determines a cohomology surjection
\begin{equation}\label{fixdecoh}
 H^*( Pic_m(X)^{\tau} \times \R P^m;\Z_2) \rightarrow  H^*(SP^{m}(X)^{\tau} ;\Z_2).
\end{equation}
\end{prop}

\begin{proof}
Let $Q$ be the quotient field of $H:= H^*(B \Z/2)$. We have a commutative diagram

$$\xymatrix{ H^*_{\Z/2}(Pic_m(X) \times\C P^m) \otimes_H  Q \ar[d]^{\cong} \ar[r] & H^*_{\Z/2}(SP^{m}(X)) \otimes_H Q \ar[d]^{\cong} \\
H^*(Pic_m(X)^{\tau} \times \R P^m ) \otimes_{\Z_2} Q \ar[r] &  H^*(SP^{m}(X)^{\tau} ) \otimes_{\Z_2} Q } $$
where the vertical isomorphisms are from the localization theorem.
The top arrow is a surjection by Proposition \ref{equivsurje}, so the bottom arrow is as well. But the bottom arrow is just an extension of scalars of the map (\ref{fixdecoh}) so it must also be a surjection. 
\end{proof}

Restricting the surjection $(\ref{fixdecoh})$ to path components yields the following corollary.

\begin{cor}\label{g+1Gens}
For every path component $P \subseteq SP^m(X)^\tau$, there exists a surjective ring homomorphism $$ \wedge(h_1,...,h_g) \otimes S(w) \twoheadrightarrow H^*(P) $$
where $h_1,...,h_g, w$ each have degree one (compare Theorem \ref{nonerientablesymm}).
\end{cor}

Pursuing these methods a bit further, it is possible to derive a formula for the Poincar\'e polynomial of $SP^m(X)^{\tau}$.  However, the use of spectral sequences obscures the multiplicative structure of the cohomology ring, so we instead take a different approach.

\subsection{Path components and cup products}

As explained in \S \ref{Topology of real curves}, the path components of $SP^m(X)^\tau$ are classified by Stiefel-Whitney classes $w_1(D)$ which assigns a value $0$ or $1$ to each real circle. Given $k,s\geq 0$ such that $k+2s =m$, denote by $P_X(k,s) \subseteq SP^{k+2s}(X)^{\tau}$ a path component for which $k$ of the real circles are sent to $1$. We call these the \emph{odd circles}.

\begin{lem}
$P_X(0,1) \cong N_g$ is a non-orientable surface with orientation double cover homeomorphic to $X$.  
\end{lem}

\begin{proof}
Clearly $P_X(0,1)$ is a 2-manifold without boundary. If $X^{\tau}$ is empty then $P_X(0,1)$ is naturally identified with $X/\tau \cong N_g$ (Proposition \ref{Nofixed}). Otherwise $X/\tau$ is a surface with one boundary component for each real circle in $X^{\tau}$ and $P_X(0,1)$ can be constructed by attaching a copy of $SP^2(S^1)$ to each boundary component.  Since $SP^2(S^1)$ is a Mobius band, we see that $P_X(0,1)$ is a non-orientable surface without boundary. Its orientation double cover is obtained by taking $X$ and fattening up each real circle to a cylinder, producing a surface homeomorphic to $X$.
\end{proof}

\begin{prop}\label{symprodeen}
The forgetful map $f: SP^s( P_X(0,1)) \rightarrow P_X(0,s)$ induces an isomorphism in $\Z_2$-cohomology.
In particular,  there exists a ring isomorphism
$$  H^*(P_X(0,s)) \cong H^*(SP^{s}(N_g)).$$
\end{prop}

\begin{proof}
It is easy to see that $f$ has degree one, so by Poincar\'e duality $f^*$ is injective. Since both rings are generated by a basis of $g+1$ elements in degree one (Corollary \ref{g+1Gens}), $f^*$ must also be surjective.  
\end{proof}

\begin{remark}
I suspect that $f$ is in fact a homotopy equivalence, but I have not been able to prove it.
\end{remark}

Next, we consider the general case $P_X(k,s)$.  Given $(X,\tau)$ a real curve of genus $g$ and select $k$ odd real circles $S^1_1,..., S^1_k$. Let $\overline{X}$ be the singular surface defined by collapsing each of the $S_i^k$ to distinct point $*_i$. Let $X'$ be the non-singular surface obtained by ``normalizing" $\overline{X}$ which has the effect of replacing each collapse point $*_i$ by a pair of points $S^0_i$. This $X'$ is a surface of genus $g-k$ if the odd circles are non-separating, and a disconnected union of two surfaces of genus $(g-k+1)/2$ if the odd circles are separating. We have a diagram
\begin{equation}\label{XXXdiagram}
 \xymatrix{ X \ar[rd] && X' \ar[ld] \\ & \overline{X} &  }
\end{equation}
restricting to diffeomorphisms 
\begin{equation}\label{bijection}
X \setminus \coprod_{i=1}^k S^1_i \cong \overline{X} \setminus \coprod_{i=1}^k *_i \cong X' \setminus \coprod_{i=1}^k S^0_i.
\end{equation} 

There are unique involutions $\bar{\tau}, \tau'$ on $\overline{X}, X'$ respectively making (\ref{XXXdiagram}) equivariant. For $0 \leq k' \leq k$ define $P_{\overline{X}}(k',s)$ to be the image of $P_X(k',s)$ under the induced map $SP^m(X) \rightarrow SP^m(\overline{X})$.   

\begin{lem}
There is natural isomorphism $ P_{\overline{X}}(0,s) \cong P_{X'}(0,s)$.
\end{lem}

\begin{proof}
$P_{\overline{X}}(0,s)$ consists of $\bar{\tau}$-symmetric sums of points in $\overline{X}$ for which there are an even number on each real circle and on each $*_i$.  Such a sum corresponds to an element in $P_{X'}(0,s)$ which simply divides the points on $*_i$ evenly between the two components of $S^0_i$. 
\end{proof}

Consider the map
\begin{equation}\label{quasifib0}
\pi: P_X(k,s) \rightarrow P_{\overline{X}}(0,s) \cong P_{X'}(0,s)
\end{equation}
obtained by projecting $P_X(k,s) \rightarrow P_{\overline{X}}(k,s)$ then subtracting one point from each $*_i$.

\begin{prop}\label{quasifib}
The map (\ref{quasifib0}) is a quasi-fibration, with homotopy fibre $P_X(k,0) \cong (S^1)^k$. 
\end{prop}

\begin{proof}
We refer to Hatcher (\cite{H} Lemma 4K.3) for the basic properties of quasi-fibrations. Define $B_q$ to be the subset of $P_{\overline{X}}(0,s)$ consisting of sums for which at least $2q$ points lie on $\coprod_{i=1}^k *_i$. Thus
$$ \emptyset = B_{s+1} \subseteq B_s \subseteq B_{s-1} \subseteq ... \subseteq B_0 =  P_{\overline{X}}(0,s)$$
is a filtration by closed subsets. We will prove that $f^{-1}(B_q) \rightarrow B_q$ is a quasi-fibration by induction on $s-q$.

First observe that $$ B_q \setminus B_{q+ 1} = \coprod_{ q_1+...+q_k = q} B_{(q_1,...,q_k)} $$
where $B_{(q_1,...,q_k)} \subseteq P_{X'}(0,s)$ consists of sums with exactly $2q_i$ points on  $*_i$ for each $i$. There is a natural homeomorphism
$$ \pi^{-1}(B_{(q_1,...,q_k)}) \cong  B_{(q_1,...,q_k)} \times SP^{2q_1+1}(S^1) \times ... \times SP^{2q_k+1}(S^1)$$
so $\pi$ restricts to a trivial fibre bundle over $B_{(q_1,...,q_k)} $. Positive symmetric powers of a $S^1$ are homotopy equivalent to $S^1$, so $\pi$ restricts to a quasi-fibration over $B_q \setminus B_{q+1}$ with homotopy fibre $(S^1)^k$. In particular,  $\pi^{-1}(B_s) \rightarrow B_s$ is a quasi-fibration.

Let $ X^{\tau} \subseteq U  \subseteq X$ be a tubular neighbourhood which $\tau$-equivariantly deformation retracts onto $X^{\tau}$.  Define $U_q \subseteq P_{X'}(0,s)$ those sums containing at least $2q$ points inside of $U$. Then $U_q$ deformation retracts onto $B_q$ lifting to a deformation retraction of $\pi^{-1}(U_q)$ onto $\pi^{-1}(B_q)$ inducing homotopy equivalences on fibres.  

Suppose inductively that $\pi$ is a quasi-fibration over $B_q$. Then the deformation retraction above implies that $\pi$ is a quasi-fibration over $U_q$ by (\cite{H} Lemma 4K.3 (c)).  Since $\pi$ also quasifibration over $B_{q-1} \setminus B_q$, $\pi$ must be a quasi-isomorphism over $B_{q-1} = (U_q \cap B_{q-1}) \cup (B_{q-1} \setminus B_q)$ by (\cite{H} Lemma 4K.3 (a)), completing the induction.
\end{proof}

\begin{prop}\label{SScoll}
Given any divisor $D_0 \in P_X(0,s)$, the map 
\begin{equation}\label{fibinc}
 i: P_X(k,0) \rightarrow P_X(k,s), ~~~~~~~D \mapsto D+D_0
 \end{equation}
induces a surjection in $\Z_2$-cohomology.  In particular, the Serre spectral sequence of (\ref{quasifib0}) collapses and we obtain an isomorphism  
\begin{equation}\label{fibtensorcoll} 
 H^*(P_X(k,s);\Z_2) \cong H^*((S^1)^k;\Z_2) \otimes H^*(P_{X'}(0,s);\Z_2 ).
 \end{equation}
of graded $H^*( P_{X'}(0,s);\Z_2)$-modules.

Set $\bar{g} = g-k$. If the odd circles are not separating then the Poincar\'e series equals 
$$ P_t(P_X(k,s)) =  (1+t)^k\sum_{i=0}^{min(s, \bar{g})} { \bar{g} \choose i }  \frac{ t^{2s-i+1}-t^i}{t-1}$$
If the odd circles are separating then 
$$ P_t(P_X(k,s)) =  (1+t)^k\sum_{i=0}^{min(s, \bar{g}+1)} { \bar{g}+1 \choose i } \frac{ t^{2s-i+2}-t^i}{t^2-1}$$ 
In either case the Euler characteristic satisfies
$$ \chi( P_X(k,s)) = \begin{cases}  0 & \text{if $k > 0$} \\ (-1)^s{ \bar{g}-1 \choose s} & \text{if $k=0$} \end{cases} $$
\end{prop}

\begin{proof} 
By Corollary \ref{g+1Gens}, $H^*( P_X(k,s))$ is generated in $(g+1)$ elements in degree one, $\alpha_1,...,\alpha_{g+1} \in H^1(P_X(k,s))$. Therefore, $\im(i^*)$ is generated by $i^*(\alpha_1),...,i^*(\alpha_{g+1}) \in H^1(P_X(k,0))$. Because $H^*(P_X(k,0))$ is an exterior algebra generated by $k$ elements in degree one, it follows that $i^*$ is surjective if and only if $H^k(P_K(k,0)) \subseteq \im(i^*)$.  

The inclusion $i$ extends to the forgetful map
$$ f: P_X(k,0) \times P_X(0,s)  \rightarrow  P_X(k, s) .$$
This is a degree one map of closed manifolds sending one fundamental class to the other. In terms of the Kunneth decomposition, this means that 
$$H^{k}(P_X(k,0)) \times  H^{2s}(P_X(0,s))  \subseteq  \im(f^*).$$

On the other hand, the image of $f^*$ is generated by the image of the generators $\alpha_1,...,\alpha_{g+1}$ which have the form
$$  f^*(\alpha_i) = i^*(\alpha_i) \otimes 1 + 1 \otimes \gamma_i $$
for some $\gamma_i \in H^*(P_X(0,s))$. Therefore, $\im(f^*) \subseteq \im(i^*) \otimes H^*( P_X(0,s))$. It follows that $\im(i^*)$ must contain $H^k(P_X(k,0))$ and thus $i^*$ is surjective.

Applying the Leray-Hirsch Theorem to the quasi-fibration of Proposition \ref{quasifib} yields the isomorphism (\ref{fibtensorcoll}). Input the formulas from Theorem \ref{McDThm} and Theorem \ref{nonerientablesymm} to compute the Poincar\'e series and Euler characteristic.
\end{proof}

\begin{remark}\label{n=o2}
If $\bar{g} = g-k \geq 0$, then $P_t(P_X(k,s))$ is independent of whether or not the odd circles are separating, due to Remark \ref{n=o}.
\end{remark}

\begin{thm}\label{Realsymmprodform}
Let $(X,\tau)$ be a real curve of genus $g$ with $n$ real circles, and let $P_X(k,s) \subseteq SP^{k+2s}(X)^{\tau}$ be a path component with $k$ odd circles. If the odd circles are not separating then
\begin{equation}\label{nonsepcoh}
 H^*(P_X(k,s)) \cong \wedge(x_1,...,x_{k}) \otimes H^*(SP^s(N_{g-k})).
 \end{equation} 
If the odd circles are separating, then 
\begin{equation}\label{sepcoh}
 H^*(P_X(k,s) ) \cong \wedge(x_1,...,x_{k-1}) \otimes H^*(SP^s(\Sigma_{(g-k+1)/2}))[\sqrt{\eta}] 
 \end{equation}
where we adjoin a square root of $\eta \in H^*(SP^s(\Sigma_{(g-k+1)/2}))$ defined in the presentation (\ref{McDrelations}).\end{thm}

\begin{proof}
If $s=0$, then $P_X(k,s)\cong (S^1)^k$ and the result holds. So assume $s\geq 1$. By Proposition \ref{SScoll}, we have quasi-fibration sequence $P_X(k,0) \stackrel{i}{\rightarrow} P_X(k,s) \stackrel{\pi}{\rightarrow} P_{X'}(0,s)$ inducing a cohomology injection $\pi^*$ and surjection $i^*$ of $H^*(P_X(k,s))$ onto $ H^*(P_X(k,0)) = H^*((S^1)^k)$ which is an exterior algebra on $k$ generators of degree one. Choose $x_1,...,x_k \in H^*(P_X(k,s))$ so that $i^*(x_1),...,i^*(x_k)$ forms a basis for $H^1(P_X(k,0))$. Then by the Leray-Hirsch Theorem, we have a surjective ring homomorphism
\begin{equation}\label{stepcoll}
 S(x_1,...,x_k) \otimes  H^*(P_{X'}(0,s)) \twoheadrightarrow H^*(P_X(k,s)). 
 \end{equation}
Recall from Corollary \ref{g+1Gens} that $H^*(P_X(k,s))$ is generated as a ring by elements of degree one, all but one of which squares to zero. Therefore, we can assume that $x_i^2 = 0$ for $i=1,...,k-1$ and the surjection (\ref{stepcoll}) factors through a surjection
$$ \wedge(x_1,...,x_{k-1}) \otimes S(x_k) \otimes  H^*(P_{X'}(0,s)) \twoheadrightarrow H^*(P_X(k,s)).  $$

If $X^{\tau}$ is non-separating, then there exists $ z \in H^1(P_{X'}(0,s)) = H^1(SP^{s}(N_{g-k}))$ such that $z^2 \neq 0$. Thus either $x_k^2=0$ or we can replace it with $x_k+\pi^*(z)$ which satisfies $(x_k+\pi^*(z))^2 = x_k^2+\pi^*(z)^2 = 0$ proving (\ref{nonsepcoh}).

If $X^{\tau}$ is separating, then every element of $H^1(P_{X'}(0,s)) = H^1(SP^s(\Sigma_{(g-k+1)/2}))$ squares to zero. It follows that $x_k^2 \neq 0$. Recall that the forgetful map $f: P_X(k,0) \times P_X(0,s) \rightarrow P_X(k,s)$ has degree one, so it induces a cohomology injection. Decompose $f^*(x_k) = y \otimes 1 +1 \otimes z$ according to Kunneth formula. Then $0 \neq f^*(x_k^2) = y^2 \otimes 1 + 1 \otimes z^2 = 1 \otimes z^2 \in 1 \otimes H^2(P_X(0,s))$. By Corollary \ref{g+1Gens}, $H^2(P_X(0,s))$ contains a unique non-zero square. Thus to prove (\ref{sepcoh}), it suffices to show that the map $P_X(0,s) \rightarrow P_{X'}(0,s)$ sends $\eta$ to this unique non-zero square.  By Proposition \ref{symprodeen} we have a commuting diagram
$$ \xymatrix{    H^*( SP^s( P_X(0,1)) ) && \ar[ll]_{\cong} H^*( P_X(0,s)) \\ 
H^*( SP^s(P_{X'}(0,1))) \ar[u]^{bd_s^*} && \ar[ll]_{\cong} H^*( P_{X'}(0,s)) \ar[u]} $$
where $bd_s$ is the symmetric power of the blow-down map considered in Lemma \ref{c to b}, completing the proof.

\end{proof}

\section{Real Abel-Jacobi and covering spaces}\label{Real Abel-Jacobi and covering spaces}

Recall from \S \ref{Topology of real curves} that if $\Sigma^{\tau}$ is a disjoint union of $n \geq 1$ real circles, then for all $l \in \Z$, $Pic_l(X)^{\tau}$ is a torsor for $Pic_0(X)^{\tau} \cong (\Z/2)^{n-1} \times (S^1)^g$. Consider the homomorphism $sq: Pic_0(\Sigma)^{\tau} \rightarrow Pic_0(\Sigma)^{\tau}$ sending the divisor class $[D]$ to $sq([D]) =[-2D]$. Then $sq$ maps to the identity component forms a $(\Z/2)^{g+n-1}$-principal covering over it. More generally, given a fixed divisor $D_0$ of degree $d$, the map $$sq: Pic_l(X)^{\tau} \rightarrow Pic_{d -2l}(X)^{\tau}$$ sending $[D]$ to $[D_0-2D]$ is a $(\Z/2)^{g+n-1}$-principal cover over the path component of $Pic_{d-2l}(X)^{\tau}$ corresponding to $w_1(D_0)$. 

Applying this to the pull-back diagram (\ref{pullbackreal}), we identify the real part of the $U(1)$-fixed point component $F^{\tau}_l$ with a pull-back of the form
\begin{equation}
 \xymatrix{  F_l^\tau \ar[r] \ar[d] &  (\Z/2)^{n-1} \times (S^1)^g \ar[d]^{sq} \\
P_X(k,s)  \ar[r]^{aj} & (S^1)^g } 
\end{equation}
where $k$ is the number of odd circles of $D$, $k+2s=m = \deg(K(D))-2l$, $(S^1)^g$ is identified with the path component of $Pic_m(X)^{\tau}$ containing $[K(D)]$, and $sq$ is the group homomorphism sending each element to its square. Clearly then, $F_l^{\tau}$ is a coproduct of $2^{n-1}$ copies of the pullback
\begin{equation}\label{simplified}
 \xymatrix{  P \ar[r] \ar[d] & (S^1)^g \ar[d]^{sq} \\
P_X(k,s)  \ar[r]^{aj} & (S^1)^g  .} 
\end{equation}
We are reduced to computing the mod 2 Betti numbers of $P$.

In general, suppose $\pi: \tilde{M} \rightarrow M$ is $(\Z/2)^q$-principal bundle determined by the $q$-tuple $(w_1,...,w_q)$ of Stiefel-Whitney classes $w_i \in H^1(M;\Z_2)$ and $M$ is a finite type CW-complex. The classifying map
$$ \xymatrix{  \tilde{M} \ar[r]  \ar[d]& (S^{\infty})^q \ar[d]\\
M \ar[r] & (\R P^{\infty})^{q}. } $$
gives rise to an Eilenberg-Moore Spectral sequence, $EM^{*,*}_*$ converging  to $H^*(\tilde{M};\Z_2)$.\footnote{The spectral sequence converges to $H^*(\tilde{M})$ because $\pi_1((\R P^{\infty})^q) =(\Z/2)^q $ is a finite $2$-group and the homotopy fibre $\tilde{M}$ has finite dimensional cohomology in each degree. See \cite{BD} citing \cite{Dw}.} The second page $EM_2^{*,*} $ equals the cohomology of a bigraded differential graded algebra

\begin{equation}\label{EM2formula}
 EM_2^{*,*}= H(\wedge(u_1,...,u_k) \otimes H^*(M;\Z_2) , \delta)
 \end{equation}
where $\deg(u_i) = (-1,1)$, $\deg(H^k(M;\Z_2)) = (0,k)$, $\delta(u_i) = w_i$ and $\delta( H^*(M;\Z_2)) = 0$.

In case $q=1$, $\tilde{M} \rightarrow M$ is a 2-fold cover and (\ref{EM2formula}) is equivalent to the associated Thom-Gysin sequence. In particular, because (\ref{EM2formula}) has only two non-zero columns when $q=1$, the spectral sequence collapses to give $P_t(EM_2^{*,*}) = P_t(\tilde{M})$. 

Introduce the relation $P_t(X) \geq P_t(Y)$ if the difference $P_t(X) -P_t(Y)$ has only non-negative coefficients.

\begin{lem}\label{coverge}
Let $\pi: \tilde{M} \rightarrow M$ be a $(\Z/2)^q$-principal bundle classified by Stiefel-Whitney classes $(w_1,...,w_q)   \in H^1(M)^{\times q}$.  If $w_i^2= 0$ for all $w_i$, then $P_t(\tilde{M}) \geq P_t(M)$. 
\end{lem}

\begin{proof}
If $q =1$ then we have the Thom-Gysin sequence:
$$... \rightarrow H^{k-1}(M) \stackrel{\phi_{k-1} }{\rightarrow}  H^k(M) \stackrel{\pi^*}{\rightarrow} H^k(\tilde{M} ) \rightarrow H^k(M) \stackrel{\phi_k}{\rightarrow} H^{k+1}(M) \rightarrow ...$$
where $\phi_k$ is cup product by $w_1$.  Because $w^2_1=0$ we have $\im(\phi_{k-1}) \subseteq \ker(\phi_k)$. The result follows by exactness.

For $q > 1$, we construct a sequence of $2$-fold covers  $\tilde{M} = M_q \rightarrow M_{q-1} \rightarrow ... \rightarrow M_0 \rightarrow M$ by taking partial quotients by subgroups of $(\Z/2)^q$. The Stiefel-Whitney class for the covering $M_{i+1} \rightarrow M_i$ is the pull-back of $w_i \in H^*(M;\Z_2)$ so it must square to zero. By induction $P_t(\ti{M}) \geq P_t(M)$.
 \end{proof}

\begin{lem}\label{tensorext}
Let $\tilde{M} \rightarrow M$ be as in Lemma \ref{coverge}. Suppose that there is a ring factorization $$ H^*(M;\Z_2) \cong \wedge(w_1,...,w_q) \otimes R $$ for some graded subring $R \leq H^*(M;\Z_2)$.  Then $P_t(\tilde{M}) = P_t(M)$. \end{lem}

\begin{proof}
By Lemma \ref{coverge} we have $P_t(\tilde{M}) \geq P_t(M)$. For the converse inequality, check that $ EM_2^{*,*} \cong \wedge(y_1,...,y_q) \otimes R$ where $y_i$ is the class represented by the cocycle $u_i w_i$. Therefore  $P_t(M) = P_t( E_2^{*,*}) \geq P_t(EM_{\infty}^{*,*}) = P_t( \tilde{M})$.
 
\end{proof}

\begin{prop}\label{Coveringspacebetti}
Consider the covering space $P \rightarrow P_X(k,s)$ defined by the pull-back diagram (\ref{simplified}). Introduce $\bar{g} = g-k$. Then $P$ has Poincar\'e series divisible by $(1+t)^k$. If the odd circles are not separating then
\begin{eqnarray*} P_t(P)/(1+t)^k  &=& P_t(P_{X'}(0,s)) + (-1)^s (2^{\bar{g}}-1) \chi( P_{X'}(0,s) )t^s  \\
&=&   \Big[\sum_{i=0}^{min(s,\bar{g})} { \bar{g} \choose i}  \frac{ t^{2s-i+1}-t^i}{t-1}\Big] +  (2^{\bar{g}} -1)  {\bar{g} - 1 \choose s}t^s.
\end{eqnarray*}
If the odd circles are separating then
\begin{eqnarray*} P_t(P)/(1+t)^k  &=& P_t(P_{X'}(0,s)) + (-1)^s (2^{\bar{g} +1}-1) \chi( P_{X'}(0,s) )t^s  \\
&=&   \Big[\sum_{i=0}^{min(s,\bar{g}+1)} { \bar{g}+1 \choose i}  \frac{ t^{2s-i+2}-t^i}{t^2-1}\Big] +  (2^{\bar{g} +1} -1)  {\bar{g} - 1 \choose s}t^s.
\end{eqnarray*}
\end{prop}

\begin{proof}
The $(\Z/2)^g$-principal bundle $P \rightarrow P_X(k,s)$ has Stiefel-Whitney classes $(w_1,...,w_g)$ given by the image of a basis $z_1,...,z_g \in H^1((S^1)^g)$ under $$aj^*_1: H^1((S^1)^g) \rightarrow H^1( P_X(k,s)).$$
Let  $q := \bar{g}+1$ if $X^{\tau}$ is separating, and $q:=\bar{g}$ if not.  

If $s=0$, then $P_X(k,0) = (S^1)^k$ and by Theorem \ref{Realsymmprodform} we may choose the basis so that $w_1=... =w_q =0$ and $w_{q+1},...,w_g \in H^1( P_X(k,0))$ are linearly independent. It follows that
$$P = \coprod_{2^q} (S^1)^{k}$$
in accord with formulas above.

Now suppose that $s\geq 1$. Then any basis $z_1,...,z_g$ will map to a linearly independent set $w_1,..,w_g$. We can choose the basis so that $w_1,...,w_q$ lie in the image of $\pi^*$ where $\pi: P_X(k,s) \rightarrow P_{X'}(0,s)$ is the quasi-fibration (\ref{quasifib0}), and $w_{q+1},...,w_g$ are sent to linearly independent  elements in the cohomology of $P_X(k,0) \cong (S^1)^k$, the quasi-fibre of $\pi$. 

Consider the intermediate covering space
$$ P \rightarrow P' \rightarrow P_X(k,s) $$
where $P'$ has is the $(\Z/2)^q$-bundle classified by $w_1,...,w_q$. Then $P'$ is a pull-back
$$\xymatrix{ P' \ar[r] \ar[d] &  P_X(k,s) \ar[d] \\   P'' \ar[r] & P_{X'}(0,s)}. $$
In particular, $P' \rightarrow P''$ is a quasi-fibration with quasi-fibre $(S^1)^{k}$, whose Serre spectral sequence collapses because it is a pull-back of (\ref{quasifib0}), so $P_t(P') = (1+t)^{k}P_t(P'')$.

The $(\Z/2)^{g-q}$-bundle $P \rightarrow P'$ is classifed by the images of $w_{q+1},...,w_g$. By Theorem \ref{Realsymmprodform}, we see that it satisfies the conditions of Lemma \ref{tensorext}, so $$P_t(P) = P_t(P') = (1+t)^{k}P_t(P'').$$
It remains to compute $P_t(P'')$. By Lemma \ref{coverge}, we know $P_t(P'') \geq P_t(P_{X'}(0,s))$. 

If $X^{\tau}$ is non-separating then by Proposition \ref{symprodeen} we have an isomorphism $ H^*(P_{X'}(0,s)) \cong H^*(SP^s(N_{q}))$. If $X^{\tau}$ is separating, then $ H^*(P_{X'}(0,s)) \cong H^*(SP^s(\Sigma_{(q/2)})$.   Recall from Corollary \ref{g+1Gens} and Proposition \ref{Nofixed2} we have a surjective homomorphism 
\begin{equation}\label{dobulesurj}
\wedge(x_1,...,x_q) \otimes S(\alpha) \rightarrow H^*( P_{X'}(0,s))
\end{equation}
which is an isomorphism in degree less than $s+1$, where $\alpha =w$ has degree one if $X^{\tau}$ is non-separating and $\alpha = \eta$ has degree two if $X^{\tau}$ is separating. The Stiefel-Whitney invariants of $P''$ are the images of $(x_1,...,x_q)$ so (\ref{dobulesurj}) extends naturally to a dga morphism
$$\wedge(u_1,...,u_q) \otimes \wedge(x_1,...,x_q) \otimes S(\alpha) \rightarrow \wedge(u_1,...,u_q) \otimes H^*( P_{X'}(0,s)).$$
where $\delta(u_i) = x_i$. This yields a map in cohomology (where $y_i := [u_ix_i]$)
$$  \wedge(y_1, ..., y_q) \otimes S(\alpha) \rightarrow E_2^{*,*}$$
which is an isomorphism total degree less than $s$. Thus
$$ P_t( P_{X'}(0,s))  =  P_t(EM_2^{*,*}) \geq P_t(EM_{\infty}^{*,*}) = P_t( P'')~mod~t^s $$
so $P_{X'}(0,s)$ and $P''$ have the same Betti numbers for degrees less than $s$.  Since they are both compact manifolds of dimension $2s$, by Poincar\'e duality they also have the same Betti numbers for degree greater than $s$, so they differ only in the middle degree $s$.  Because the Euler characteristics satisfy $\chi( P'') = 2^{q} \chi(P_{X'}(0,s))$, this forces $$\dim( H^s(P'')) = \dim(H^s(P_{X'}(0,s))) +  (-1)^s(2^{q} -1)\chi(P_{X'}(0,s)).$$
It only remains to input the formulas from Proposition \ref{SScoll}.

\end{proof}

\section{Betti numbers of the moduli space}\label{Higgs bundle Poincare series}

\subsection{Betti numbers of $M(2,D)^{\tau}$}

\begin{thm}\label{MrDBetti}
Let $(\Sigma,\tau)$ be a real curve of genus $g \geq 2$ with $n$ real circles, let $D$ be a real divisor of odd degree with $k$ odd circles. Let $\bar{g} = g-k$ and $b = n-1$. 

If the odd circles are not separating then the Poincar\'e series of $M(2,D)^{\tau}$ is 
\begin{eqnarray*}
 P_t(M(2,D)^{\tau})  &= &  \frac{ (1+t)^{b} (1+t^2)^{b}(1+t^3)^{g-b} - 2^bt^g(1+t)^g}{(1-t)(1-t^2)} \\
 &+& 2^b(1+t)^k \sum_{s = 0 }^{ g - (k+3)/2} \Big( (2^{\bar{g}} -1)  {\bar{g} - 1 \choose s}t^{s} + \sum_{i=0}^{min(s ,\bar{g})} { \bar{g} \choose i} \frac{t^{2s-i+1} -t^i}{t-1} \Big)t^{3g-3-k-2s}
\end{eqnarray*}

If the odd circles are separating, then 
\begin{eqnarray*}
 P_t(M(2,D)^{\tau})  &= &  \frac{ (1+t)^{b} (1+t^2)^{b}(1+t^3)^{g-b} - 2^bt^g(1+t)^g}{(1-t)(1-t^2)} \\
 &+& 2^b(1+t)^k \sum_{s = 0 }^{ g - (k+3)/2} \Big( (2^{\bar{g} +1} -1)  {\bar{g} - 1 \choose s}t^{s} + \sum_{i=0}^{min(s ,\bar{g}+1)} { \bar{g}+1 \choose i} \frac{t^{2s-i+2} -t^i}{t^2-1} \Big)t^{3g-3-k-2s}
\end{eqnarray*}
\end{thm} 

\begin{remark}
The range of possible inputs into Theorem \ref{MrDBetti} is as follows (see \S \ref{Topology of real curves}). If the odd circles are not separating, then $k \equiv 1~mod~2$ and either $ 1 \leq k \leq n \leq g$ or $1 \leq k < n = g+1$. If the odd circles are separating, then $1 \leq k = n \leq g+1$ and $g \equiv n+1 \equiv 0~mod~2$.  
\end{remark}

\begin{remark}\label{lastremark}
If $(g,n,k) = (2,3,3)$ (so the odd circles must separate), then the second line in the above formula is empty, because $(g-(k+3)/2) < 0$.  In this case, there are no Morse critical points except the global minimum $N(2,D)^{\tau}$ so the Morse flow deformation retracts $M(2,D)^{\tau}$ onto $N(2,D)^{\tau}$, giving $P_t(M(2,D)^{\tau}) = 1+3t +3t^2+t^3$.
\end{remark}

\begin{proof}[Proof of Theorem \ref{MrDBetti}]
Given a real divisor $D'$, there is a natural isomorphism $$M(2, D)^{\tau} \cong M(2,D+2D')^{\tau}$$ defined by tensoring by $\mathcal{O}(D')$.  Since $\deg(D+2D') = \deg(D) +2\deg(D')$ and $w_1(D) = w_1(D+2D')$, we may assume without loss of generality that $\deg(D) =1$.

Recall (\ref{PPMform}), adding in Morse indices $d_F = g+2l-2$,
\begin{equation}\label{tempformula}
 P_t(M(2,D)^{\tau}) = P_t(N(2,D)^{\tau}) + \sum_{l = 1}^{g-1} P_t(F_l^{\tau}) t^{g+2l-2}
\end{equation}
where $N(2,D)^{\tau}$ is the moduli space of Real bundles of rank 2 and determinant $D$. By \cite{B2} we have
$$ P_t(N(2,D)^{\tau}) = \frac{ (1+t)^{b} (1+t^2)^{b}(1+t^3)^{g-b} - 2^bt^g(1+t)^g}{(1-t)(1-t^2)}.  $$

The remaining terms were calculated in \S \ref{Real Abel-Jacobi and covering spaces}.  If $ m := 2g-2l-1 < k$ then $P_l^{\tau} = \emptyset$. If $m\geq k$ we have 
$$ P_t(F_l^{\tau}) =  2^bP_t(P)$$
where $P$ is defined in Proposition \ref{Coveringspacebetti}, and $2s:= m-k$. We replace the index $l$ by $s = g-l -(k+1)/2$, which ranges from $s = 0$ to $s=g - (k+3)/2$.  Inputting into (\ref{tempformula}) completes the proof.
\end{proof}

\begin{remark}
The Poincar\'e series $P_t(M(2,D)^{\tau})$ is sensitive to whether or not the odd circles are separating. For example, with inputs $(g,n,k) = (2,1,1)$ we get 
$$P_t(M(2,D)^{\tau}) = \begin{cases}  (1+t)(1+3t^2) & \text{ if non-separating} \\   (1+t)(1+5t^2) & \text{ if separating} \end{cases}$$ 
\end{remark}

\subsection{Betti numbers of $M(2,d)^{\tau}$}

Let $(X,\tau)$ be a real curve of genus $g$ with $n\geq 1$ real circles, and consider the moduli of rank $r$, odd degree $d$ stable Higgs bundles $M(r,d)$ without fixing a determinant ($gcd(r,d) =1$ as always). The fixed point set $M(r,d)^{\tau}$ decomposes into $2^{n-1}$ path components
$$M(r,d)^{\tau} =  \coprod_{w} M(r,d)^{\tau}_w  $$
classified by $w \in H^1(X^{\tau};\Z_2)$ the first Stiefel-Whitney class of the determinant line bundle, which must satisfy $w(X^{\tau}) \equiv d~mod~2$. 

Tensor product defines a natural isomorphism $$M(r,d) \cong M(r,D) \times_{\Gamma_{2g}} M(1,0)$$ where $\Gamma_{2g} \cong (\Z/r)^{2g}$ is the $r$-torsion subgroup of $M(1,0) = Pic_0(X)\times H^0(X,K)  \cong (S^1)^{2g} \times \C^g$.  If we choose a real divisor $D$ such that $w_1(D)= w$, then we similarly obtain an isomorphism
$$ M(r,d)^{\tau}_w \cong M(r,D)^{\tau} \times_{\Gamma_g} M(1,0)^{\tau}_0$$
where $\Gamma_g \cong (\Z/r)^g$ is the $r$-torsion subgroup of $M(1,0)^{\tau}_0 \cong (S^1)^g \times \R^g$.

\begin{thm}\label{M(2,d)Betti}
Let $(\Sigma,\tau)$ be a real curve of genus $g \geq 2$ with $n$ real circles, let $D$ be a real divisor of odd degree $d$, and let $w = w_1(D)$. Let $\bar{g} = g-k$ and $b = n-1$. 

If the odd circles are not separating then the mod 2 Poincar\'e series of $M(2,d)^{\tau}_w$ is 
\begin{eqnarray*}
 P_t(M(2,d)^{\tau}_w)  &= &  \frac{ (1+t)^{b+g} (1+t^2)^{b}(1+t^3)^{g-b} - 2^bt^g(1+t)^{2g}}{(1-t)(1-t^2)} \\
 &+& 2^b(1+t)^{k+g} \sum_{s= 0 }^{ g - (k+3)/2} \Big( \sum_{i=0}^{min(s ,\bar{g})} { \bar{g} \choose i} \frac{t^{2s-i+1} -t^i}{t-1} \Big)t^{3g-3-k-2s}
\end{eqnarray*}

If the odd circles are separating, then 
\begin{eqnarray*}
 P_t(M(2,d)^{\tau}_w)  &= &  \frac{ (1+t)^{b+g} (1+t^2)^{b}(1+t^3)^{g-b} - 2^bt^g(1+t)^{2g}}{(1-t)(1-t^2)} \\
 &+& 2^b(1+t)^{k+g} \sum_{s = 0 }^{ g - (k+3)/2} \Big(  \sum_{i=0}^{min(s ,\bar{g}+1)} { \bar{g}+1 \choose i} \frac{t^{2s-i+2} -t^i}{t^2-1} \Big)t^{3g-3-k-2s}
 \end{eqnarray*}

\end{thm}

\begin{proof}
The discussion in \S \ref{Outline of the proof} applies analogously to $M(2,d)^{\tau}$. In particular, the function $\mu: M(2,d) \rightarrow \R$, $\mu(E,\Phi) = \| \Phi \|_{L^2}^2$ restricts to a proper, $\Z_2$-perfect Morse-Bott function on $M(2,d)^{\tau}$, hence also on its path components $M(2,d)_w^{\tau}$,  with critical set equal to $M(2,d)^{\tau}_w \cap M(2,d)^{U(1)}$. The global minimum of $\mu$ on $M(2,d)^{\tau}_w$ is equal to the moduli space $N(2,d)^{\tau}_w$ of stable Real vector bundles whose Betti numbers were computed in \cite{B, LS}:
$$P_t( N(2,d)^{\tau}_w) = (1+t)^g P_t(N(2,D)^{\tau}).$$ 
The remaining critical loci have the form $\coprod_{2^b} Q$ where 
$$Q:= P \times_{\Gamma_g}  M(1,0)^{\tau}_0$$ 
and $P \rightarrow P_X(k,s)$ is the $\Gamma_g$-principal cover (\ref{simplified}) for some $s$. Therefore, projecting onto the first factor determines a fibre bundle
$$(S^1)^g \times \R^g \rightarrow Q \rightarrow P_X(k,s).$$
Since $\Gamma_g$ acts on $M(1,0)^{\tau}_0 \cong (S^1)^g\times \R^g$ by translation, $\pi_1(P_X(k,s))$ acts trivially on the cohomology of the fibre, and the Serre spectral sequence gives us an inequality
$$P_t(Q) \leq (1+t)^g P_t( P_X(k,s)). $$
On the other hand, $Q$ is a $\Gamma_g$-principal cover of $$ (P/\Gamma_g) \times (M(1,0)^{\tau}_0/\Gamma_g) \cong P_X(k,s) \times (S^1)^g\times \R^g$$
determined by Stiffel-Whitney classes that square to zero, so applying Lemma \ref{coverge} we get the converse inequality and conclude that $P_t(Q) = (1+t)^g P_t(P_X(k,s))$. We obtain
\begin{equation}\label{lastone}
 P_t( M(2,d)^{\tau}_w) = (1+t)^g \Big[ P_t( N(2,D)^{\tau}) + \sum_{s = 0 }^{ g - (k+3)/2}  2^{b}P_t(P_X(k,s)) t^{3g-3-k-2s} \Big].
 \end{equation}
\end{proof}

\begin{cor}
The Poincar\'e series $P_t(M(2,d)^{\tau})$ is independent of whether or not the odd circles are separating. 
\end{cor}

\begin{proof}
If both separating and non-separating examples exist for given $(g, n, k)$, then $g\geq n= k$,  so we can apply Remark \ref{n=o2} to (\ref{lastone}).
\end{proof}

\section*{Acknowledgement}

Thanks to Steve Rayan and Noah MacAulay for helpful discussions and to my host Lisa Jeffrey at the University of Toronto where this research was conducted. Thanks also to Matthias Franz for helpful comments on an earlier draft. This research was supported by an NSERC Discovery Grant.

\end{document}